\documentclass[11pt]{huber_article}

\usepackage{amsmath,amssymb,amsthm}
\usepackage{booktabs}
\usepackage{comment}

\newcommand{\lowerbound}
  {(1/5)\epsilon^{-2}(1 + 2\epsilon)(1 - \delta)\ln((2-\delta)\delta^{-1})p^{-1}}

\newcommand{\real}{\mathbb{R}}
\newcommand{\mean}{\mathbb{E}}
\newcommand{\prob}{\mathbb{P}}

\newcommand{\unif}{\textsf{Unif}}
\newcommand{\gammadist}{\textsf{Gamma}}
\newcommand{\iid}{\mathrel{\mathop{\sim}\limits^{\textrm{iid}}}}
\newcommand{\ex}{\textsf{Exp}}
\newcommand{\bern}{\textsf{Bern}}
\newcommand{\geo}{\textsf{Geo}}
\newcommand{\gbas}{\text{\sffamily{GBAS}}}
\newcommand{\dklr}{\text{\sffamily{DKLR}}}
\newcommand{\ind}{{\bf 1}}

\begin{document}

\newtheorem{theorem}{Theorem}
\newtheorem{lemma}{Lemma}
\newtheorem{fact}{Fact}
\theoremstyle{definition}
\newtheorem{definition}{Definition}

\title{An unbiased estimate for the probability of heads on a coin where the
   relative error has a distribution independent of the coin}

\author{Mark Huber}
\email{mhuber@cmc.edu}

\maketitle

\begin{abstract}
Say $X_1,X_2,\ldots$ are 
independent identically distributed Bernoulli random variables
with mean $p$, so $\prob(X_i = 1) = p$ and $\prob(X_i = 0) = 1 - p$.
Any estimate $\hat p$ of $p$ has relative error $\hat p / p - 1$.
This paper builds a new estimate $\hat p$ of $p$
such that the relative error
of the estimate does not depend in any way on the value of $p$.  
This allows the easy construction of exact confidence intervals for $p$
of any desired level without needing any sort of limit or 
approximation.  In addition, $\hat p$ is unbiased.
The expected number of Bernoulli draws used by the algorithm is at
most 1 more than $1 - p$ times the number of draws needed if the Central
Limit Theorem held exactly.
For $\epsilon$ and $\delta$ in $(0,1)$, 
to obtain an estimate where $\prob(|\hat p/p - 1| > \epsilon) \leq \delta$,
the new algorithm takes on average at most 
$2\epsilon^{-2} p^{-1}\ln(2\delta^{-1})(1 - (4/3) \epsilon)^{-1}$ 
samples.  It is also
shown that any such algorithm that applies whenever 
$p \leq 1/2$ requires at least 
$\lowerbound$
samples.  The same algorithm can also be applied to estimate the mean of 
any random variable that falls in $[0,1]$.
Applications of this methodology include finding exact $p$-values and
estimating normalizing constants and Bayes' Factors using
acceptance/rejection.
\end{abstract}

\section{Introduction}

Say $X_1,X_2,X_3,\ldots$ are independent, identically distributed (iid)
Bernoulli random variables with mean $p$.  Write $X_i \sim \bern(p)$ to
denote
$\prob(X_i = 1) = p$ and $\prob(X_i = 0) = 1 - p$.
The purpose of this work is to present a new algorithm for estimating
$p$ with $\hat p$ so that the relative error $\hat p/p - 1$ has a known
distribution that does not depend on the value of $p$.  In other words,
with this
algorithm it is
possible to compute $\prob(a \leq \hat p/p  - 1 \leq b)$ exactly for any
$a \leq 0 \leq b$, without needing any kind of approximation or
limiting behavior.

This problem of estimating $p$, 
which is also known as estimating the parameter of a
binomial given a large sample, arises in a wide diversity of contexts.  
Examples include
estimating the percentage of farms growing a 
particular crop~\cite{mahalanobis1940}, estimating the prevalence of a 
disease in a population~\cite{rogang1978,rahmejg2000}, and any situation
where it is desirable to know the percentage of a population with a 
specific property.

Another application is in exact $p$-values.  Given a statistical model
and a statistic, let ``heads'' be when the statistic applied to a draw
from the model is more unusual than the same statistic applied to the 
data, and all other events are ``tails.''  Then the $p$-value for the
data is just the probability of heads on the coin.
This allows estimation of the exact $p$-value for any statistical model
that can be simulated from by flipping coins.  Models where this
has been applied include testing if a population is in Hardy-Weinberg
equilibrium~\cite{galbuseralswm2000,huber2006c}, the Rasch 
model~\cite{besagc1989}, two-sample survival data~\cite{senchaudhurimp1995},
and nonparametric testing of sequential 
association~\cite{bakemanrq1996}.

In theoretical computer science, many problems of approximation can
be reduced to the problem of estimating the binomial parameter.  In particular,
approximating the permanent
of a matrix with positive entries~\cite{huber2008a}, the number of 
solutions to a disjunctive normal form expression~\cite{karplm1989},
the volume of a convex body~\cite{lovaszv2003}, 
estimating exact $p$-values for a model (see for instance~\cite{huber2006c})
and counting the lattice points
inside a polytope can all be put into this framework.  
In general, anywhere an acceptance rejection 
method is used to build an approximation algorithm, this problem
arises.

The cost here is usually dominated by the number of 
$\bern(p)$ flips of the coin that are needed, and so the focus here 
is on minimizing the expected number of such flips needed.

\begin{definition}
Suppose $\cal A$ is a function of $X_1,X_2,\ldots \iid \bern(p)$ and auxiliary 
randomness (represented by $U \sim \unif([0,1])$) that outputs $\hat p$.
Let $T$ be a stopping time with respect to the natural filtration so that
the value of $\hat p$ only depends on $U$ and $X_1,\ldots,X_T$.  Then call 
$T$ the {\em running time} of the algorithm.
\end{definition}

The simplest algorithm for estimating $p$ just fixes $T = n$, and sets
\[
\hat p_n = \frac{X_1 + X_2 + \cdots + X_n}{n}.
\]
In this case $\hat p_n$ has a binomial distribution with
parameters $n$ and $p$.  The standard deviation of $\hat p_n$ is
$\sqrt{p(1 - p)/n}.$  Therefore, to get an estimate which is close to
$p$ in the sense of having small relative error, 
$k$ should be of the form $C/p$ (for some constant $C$)
so that the standard deviation is $p \sqrt{(1 - p)/C}$ and so roughly
proportional to $p$.  From the Central Limit Theorem, roughly 
$2\epsilon^{-2}\ln(2/\delta)/p$ samples are necessary to get 
$\hat p_n / p \in [1-\epsilon,1+\epsilon]$ for $\epsilon \in (0,1)$.
(See Section~\ref{SEC:lowerbound} for a more detailed form of this argument.)
But $p$ is unknown at the beginning of the 
algorithm!

Dagum, Karp, Luby and Ross~\cite{dagumklr2000} dealt with this 
circularity problem with
their stopping rule algorithm.  In this context of $\bern(p)$ random 
variables, their algorithm can be written as follows.  

Fix  $(\epsilon,\delta)$ with $\epsilon \in (0,1)$ and $\delta > 0$.  
Let $T$ be the smallest integer such that 
$X_1 + \cdots + X_T \geq 1 + (1 + \epsilon)4(e - 2)\ln(2/\delta)\epsilon^{-2}$.
Then $\hat p_{\dklr} = (X_1 + \cdots + X_T)/T.$  

Call this method \dklr{}.
They showed the following 
result for their estimate (Stopping Rule Theorem of~\cite{dagumklr2000}).
\begin{equation}
\prob(1 - \epsilon \leq \hat p_{\dklr}/p \leq 1 + \epsilon) > 1 - \delta,
\end{equation}
and $\mean[T] \leq [1 + (1 + \epsilon)4(e - 2)\ln(2/\delta)\epsilon^{-2}]/p$.

They also showed that any such $(\epsilon,\delta)$ 
approximation algorithm that applies to all
$p \in [0,1/2]$ (Lemma 7.5 of~\cite{dagumklr2000}) must satisfy
\[
\mean[T] \geq (4e^2)^{-1}(1 - \delta)(1 - \epsilon)^2
 (1 - p)\epsilon^{-2}\ln(\delta^{-1}(2 - \delta)).
\]

There are several drawbacks to \dklr{}.  
First, the factor of $4(e - 2)$ (which is about 2.873) in the running time of
\dklr{} is an artifact of the analysis rather than coming from 
the problem itself.
As mentioned before, a 
heuristic Central Limit Theorem argument (see Section~\ref{SEC:lowerbound})
indicates that the correct factor in the running time should be 2.
Second, the \dklr{} estimate is biased.

Our algorithm has a form similar to \dklr{}, but with a continuous 
modification that
yields several desirable benefits.  The 
\dklr{} estimate $(X_1 + \cdots + X_T)/T$ is a fixed integer
divided by a negative binomial random variable.  In the algorithm proposed
here, the estimate is a fixed integer divided by a Gamma 
random variable.  Since Gamma random variables are scalable, the relative
error of the estimate does not depend on the value of $p$.  

This allows a much tighter analysis of the error, since the value of $p$
is no longer an issue.  In particular, the algorithm attains (to first order)
the $2\epsilon^{-2}p^{-1}\ln(2\delta^{-1})$ 
running time that is likely the best possible.  
The new algorithm is called the 
Gamma Bernoulli approximation scheme (\gbas{}).

\begin{theorem}
\label{THM:upperbound}
The \gbas{} method of Section~\ref{SEC:algorithm}, for any integer
$k \geq 2$, outputs an 
estimate $\hat p$ using $T$ samples where
$\mean[T] = k/p$, $\mean[\hat p] = p$, and 
$p/\hat p$ has a Gamma distribution with shape parameter $k$ and 
rate parameter $k - 1$.  The density of the relative error 
$\hat p / p - 1$ is 
\[
\frac{(k-1)^k}{(k-1)!}\cdot \frac{\exp(-(k-1)/(s+1))}{(s + 1)^{k+1}} 
 \text{ for } s \geq -1.
\]
In particular, for $\epsilon \in (0,3/4)$, $\delta \in (0,1)$, and 
\[
k = \lceil 2\epsilon^{-2} p^{-1}\ln(2\delta^{-1})(1 - (14/3) \epsilon)^{-1} 
    \rceil,
\]
then $\prob(-\epsilon \leq (\hat p/p) - 1 \leq \epsilon) > 1 - \delta$.

\end{theorem}

To understand the effectiveness of the new estimate, suppose that in fact
the value of $p$ was known exactly.  Then for a given $n$, the
probability that the relative error was at least $\epsilon$ could be
calculated exactly, and the smallest value of $n$ that makes this
probability below $\delta$ would be used.  The table below presents
to four significant digits the number of samples used by the new algorithm,
by DKLR and by
using the optimal value for $n$ assuming that $p$ was known ahead of time.
The final column gives the expected number used by the new method 
over the number needed by the exact binomial approach.

\begin{center}
\begin{tabular}{cccccc}
\toprule
$(\epsilon,\delta)$ & $p$ & New method & \dklr{} & Exact Bin.
 & New/Exact
 \\
\midrule
$(0.1,0.05)$    & 0.05 & 7700 & 23340 & 7299 & 1.067 \\
$(0.1,0.05)$    & 0.01
  & $3.850 \cdot 10^4$ & $11.67 \cdot 10^4$ & 
  $3.755 \cdot 10^4$ & 1.025 \\
$(0.1,10^{-6})$  & 0.05 & $5.122 \cdot 10^4$ & $9.174 \cdot 10^4$ &
  $4.551 \cdot 10^4$ & 1.125 \\
$(0.1,10^{-6})$  & 0.01 & $2.561 \cdot 10^5$ & $4.587 \cdot 10^5$ &
  $2.375 \cdot 10^5$ & 1.078 \\
$(0.01,0.05)$   & 0.05 & $7.683 \cdot 10^5$ & $21.41 \cdot 10^5$
 & $7.280 \cdot 10^5$ & 1.055 \\
$(0.01,0.05)$   & 0.01 & $3.842 \cdot 10^6$ & $10.70 \cdot 10^6$ &
  $3.795 \cdot 10^4$ & 1.012 \\
$(0.01,10^{-6})$ & 0.05 & $4.790 \cdot 10^{6}$ & $8.240 \cdot 10^{6}$ &
  $4.545 \cdot 10^{6}$ & 1.054\\
$(0.01,10^{-6})$ & 0.01 & $2.395 \cdot 10^{7}$ & $4.210 \cdot 10^{7}$ 
 & $2.369 \cdot 10^{7}$ & 1.011 \\
\bottomrule
\end{tabular}
\end{center}

It is important to note that the exact binomial column is not
an actual algorithm.  This is because to use this would require the
knowledge of the exact value of $p$, which is the unknown that we
are trying to find.  In some sense, this represents the optimal number
of draws possible necessary to achieve $(\epsilon,\delta)$ performance.
The fact that the running time of the new estimate comes so close to the
optimal number of draws
without needing to know $p$ is one of the great strengths of
the new approach.

In~\cite{dagumklr2000} a lower bound for the number of samples that 
any method would require was given in the more general case of 
$[0,1]$ random variables.  For 
$\{0,1\}$ random variables, this can be improved.  
The following theorem is proved in 
Section~\ref{SEC:lowerbound}.

\begin{theorem}
For $\epsilon > 0$ and $\delta \in (0,1)$ any algorithm that returns
$\hat p$ for $p \in [0,1/2]$ satisfying 
$\prob(-\epsilon\leq (\hat p/p)  - 1 \leq \epsilon) > 1 - \delta$
must have
\[
\mean[T] \geq \lowerbound.
\]
\end{theorem}

As $\epsilon$ and $\delta$ go to 0, 
the ratio between the upper and lower bounds
converges to 10 for these results.  From Central Limit Theorem considerations,
it is likely that the upper bound constant of 2 is the correct one 
(see~Section~\ref{SEC:lowerbound}).

\section{The \gbas{} Algorithm}
\label{SEC:algorithm}

The algorithm is based upon properties of a one dimensional Poisson point
process.  Write $\ex(\lambda)$ for the exponential distribution
with rate $\lambda$ and mean $1/\lambda$.  So $A \sim \ex(\lambda)$ has density
$f_A(t) = \lambda \exp(-\lambda t) \cdot \ind(t \geq 0)$.  Here
$\ind(\text{expression})$ denotes the indicator function that evaluates
to 1 if the expression is true and is 0 otherwise.

Let $A_1,A_2,\ldots$ be iid $\ex(\lambda)$ random variables.  Set
$T_i = A_1 + \cdots + A_i$.  Then $P = \{T_i\}_{i=1}^\infty$ is a 
{\em one dimensional Poisson point process of rate $\lambda$.}

The sum of exponential random variables is well known to be a Gamma 
distributed random variable.  (It is also called the Erlang distribution.)
For all $i$,
the distribution of $T_i$ is Gamma with shape parameter $i$ and rate 
parameter $\lambda$.
The density of this random variable is
\[
f_{T_i}(t) = [(i-1)!]^{-1} \lambda^i t^{i - 1} \exp(-t \lambda) \ind(t \geq 0).
\]
Write $T_i \sim \gammadist(i,\lambda)$.

The key property used by the algorithm is {\em thinning} where each 
point in $P$ is retained independently with probability $p$.  The
result is a new Poisson point process $P'$ which has rate $\lambda p$.
(See for instance~\cite[p. 320]{resnick1992}.)  

The intuition is as follows.  For a Poisson point process of rate $\lambda$,
the chance that a point in $P$ lies in an interval
$[t,t + h]$ is approximately $\lambda h$, while the chance that a point in
$P'$ lies in interval $[t,t+h] = \lambda p h$ since points are only retained
with probability $p$.  Hence the new rate is $\lambda p$.

For completeness the next lemma
verifies this fact directly 
by establishing that the distribution of the minimum
point in $P'$ is $\ex(\lambda p)$.

\begin{lemma}
\label{LEM:thinning}
Let $G \sim \geo(p)$ so for $g \in \{1,2,\ldots\}$, 
$\prob(G = g) = (1 - p)^{g - 1}p$.  Let
$A_1,A_2,\ldots \iid A$ where $A \sim \ex(\lambda).$  Then
\[
A_1 + A_2 + \cdots + A_G \sim \ex(\lambda p).
\]
\end{lemma}

\begin{proof}
$G$ has moment generating function 
$M_G(t) = \mean[\exp(-t G)] = pe^t/(1 - (1 - p)e^t)$ when $t < -\ln(1 - p)$.
The moment generating function of $A$ is 
$M_{A}(t) = \mean[\exp(-t A)] = \lambda(\lambda - t)^{-1}$ when $t < \lambda$.  
The moment generating function of $A_1 + \cdots + A_G$ is the composition
\[
M_G(\ln(M_A(t))) = 
 \frac{p \lambda(\lambda - t)^{-1}}{1 - (1 - p)\lambda(\lambda-t)^{-1}}
  = \frac{p\lambda}{p\lambda - t}, 
\]
when $t < p\lambda$, and so $A_1 + \cdots + A_G \sim \ex(\lambda p)$.
\end{proof}

Another useful fact is that exponential distributions (and so Gamma
distributions) scale easily.

\begin{lemma}
\label{LEM:scaling}
Let $X \sim \gammadist(a,b)$.  Then for $\beta \in \real$, $\beta X \sim
 \gammadist(a,\beta^{-1} b)$.
\end{lemma}

\begin{proof}
The moment generating function of $X$ is
$M_X(t) = [b/(b - t)]^a$ for $t < b$, so that of $\beta X$ is
\[
\mean[\exp(-t \beta X)] = M_X(\beta t) = [b/(b - \beta t)]^a
 = [\beta^{-1} b/(\beta^{-1}b - t)]^a
\]
exactly the moment generating function of a $\gammadist(a,\beta^{-1}b)$.
\end{proof}

Together these results give the \gbas{} approach.
\begin{center}
\begin{tabular}{rl}
\toprule
\multicolumn{2}{l}{\gbas{} \  \ {\em Input: } $k$} \\
\midrule
1) & $S \leftarrow 0$, $R \leftarrow 0$. \\
2) & Repeat \\
3) & \hspace*{1em} $X \leftarrow \bern(p)$, $A \leftarrow \ex(1)$ \\
4) & \hspace*{1em} $S \leftarrow S + X$, $R \leftarrow R + A$ \\
5) & Until $S = k$ \\
6) & $\hat p \leftarrow (k  - 1)/R$ \\
\bottomrule
\end{tabular}
\end{center}

\begin{lemma}
The output $\hat p$ of \gbas{} satisfies
\[
\frac{p}{\hat p} \sim \gammadist(k,k - 1),
\]
and $\mean[\hat p] = p$.
The number of $\bern(p)$ calls $T$ in the algorithm satisfies $\mean[T] = k/p$.
The relative error $(\hat p / p) - 1$ has density
\[
f(s) = \frac{(k-1)^k}{(k-1)!} \frac{\exp(-(k-1)/(s+1))}{(s + 1)^{k+1}}
 \text{ for } s \geq -1.
\]
\end{lemma}

\begin{proof}
From Lemma~\ref{LEM:thinning}, the distribution of $R$ is 
$\gammadist(k,p)$.  From Lemma~\ref{LEM:scaling}, the distribution
of $p/\hat p = pR/(k-1)$ is $\gammadist(k,k-1)$.  Hence 
$\mean[\hat p] = \mean[p/X]$ where $X \sim \gammadist(k,k-1)$.  Now
\begin{align*}
\mean[1/X] &= \int_0^\infty \frac{1}{s}\frac{(k-1)^k}{(k-1)!}
  s^{k - 1}\exp(-(k-1)s) \ ds \\
 &= \frac{(k-1)^k}{(k-1)!} \int_0^\infty s^{k - 2}\exp(-(k-1)s) \ ds \\
 &= \frac{(k-1)^k}{(k-1)!} \cdot \frac{(k-2)!}{(k-1)^{k-1}}
 = \frac{k-1}{k-1} = 1,
\end{align*}   
so $\mean[\hat p] = \mean[p/X] = p$.

Since $T$, the number of $\bern(p)$ drawn by the algorithm, 
is the sum of $k$ geometric random variables (each with mean $1/p$), 
$T$ has mean $k/p$.

The density of $(\hat p/p) - 1$ follows from the fact that $p / \hat p$
has a $\gammadist(k,k-1)$ distribution.
\end{proof}

Note that for given $k$ and $a$, $\prob(\hat p/p \leq a)$  
can be computed exactly in 
$\Theta(k)$ floating point operations 
using the incomplete gamma function.  Hence for a given
error bound and accuracy requirement, it is possible to exactly
find the minimum $k$ using less work than flipping $k/p$ coins.

Suppose the user desires the absolute 
relative error to be greater than $\epsilon$ with probability at most
$\delta$.  The easiest way to compute this is to note
\[
\prob(|\hat p/p - 1| > \epsilon) = \prob(p/\hat p < (1 + \epsilon)^{-1}
 \text{ or } p/\hat p > (1 - \epsilon)^{-1}).
\]
Now $p/\hat p \sim \gammadist(k,k-1)$, so it remains to find the smallest
value of $k$ that works for given $\epsilon$ and $\delta$.

For instance, if $\epsilon = 0.1$ (so $p$ is desired to one significant figure)
and $\delta = 0.05$, then $k = 388$ is the smallest value that provides
the guarantee.  Hence $388/p$ is the expected running time (see the
table in the introduction.)

\section{Upper bounds on $k$}
\label{SEC:upperbound}

Suppose a user wants to find $k$ so that
\[
\prob(a \leq \hat p/p - 1 \leq b) \geq c.
\]
Then since
\[
\prob(a \leq \hat p/p - 1 \leq b) = \prob((1 + a)^{-1} \geq p /\hat p
  \geq (1 + b)^{-1}),
\]
and $p/\hat p \sim \gammadist(k,k-1)$, it suffices to find the smallest
value of $k$ such that for $X \sim \gammadist(k,k-1)$, 
$\prob((1 + a)^{-1} \geq X \geq (1 + b)^{-1}) \geq c.$
This is how the values in the table in the 
introduction where computed.

That being said, it is useful to have a simple function $f(\epsilon,\delta)$,
such that if $k \geq f(\epsilon,\delta)$ and $X \sim \gammadist(k,k-1)$,
$\prob((1-\epsilon)^{-1} \geq X \geq (1 + \epsilon)^{-1}) > 1 - \delta$.
In particular, such a function exists for $\dklr{}$, and having such
a function for \gbas{} allows a comparison of the time needed for the two
methods.

Building such an $f$ requires theoretical bounds on the tail of a Gamma
random variable.  Chernoff bounds~\cite{chernoff1952} 
are one way to get these bounds.

\begin{fact}[Chernoff bounds] 
Let $X_1,X_2,\ldots$ be iid random variables with finite mean and 
finite moment generating function
for $t \in [a,b]$, where $a \leq 0 \leq b$.  Let $\gamma \in (0,\infty)$,
and $h(\gamma) = \mean[\exp(tX)] / \exp(t \gamma \mean[X]).$  Then
\begin{align*}
\prob(X \geq \gamma \mean[X]) &\leq h(\gamma)
 \ \ \, \text{ for all } t \in [0,b] \text{ and } \gamma \geq 1.\\
\prob(X \leq \gamma\mean[X]) &\leq
      h(\gamma)
 \ \ \, \text{ for all } t \in [a,0] \text{ and } \gamma \leq 1.
\end{align*}
\end{fact}

\begin{lemma}
For $X \sim \gammadist(k,k - 1)$, let 
$g(\gamma) = \gamma/\exp(\gamma - 1)$.  Then  
\begin{align*}
\prob(X \geq \gamma \mean[X]) &\leq g(\gamma)^k \ \ \ 
  \text{ for all } \gamma \geq 1 \\ 
\prob(X \leq \gamma \mean[X]) &\leq g(\gamma)^k \ \ \ 
  \text{ for all } \gamma \leq 1.
\end{align*}
\end{lemma}

\begin{proof}
For $X \sim \gammadist(k,k-1)$, $\mean[X] = k/(k-1)$ and 
the moment generating function is $\mean[\exp(tX)] = (1 - t/(k-1))^{-k}$ 
when $t < k - 1$.  Letting 
$\alpha = t/(k-1)$, that makes $h(\gamma)$ from the Chernoff bound
\[
h(\gamma) = \frac{(1 - \alpha)^{-k}}{\exp(\alpha k \gamma)}.
\]
Letting $\alpha = 1 - 1/\gamma$ minimizes the right hand side, making it
\[
[\gamma / \exp(\gamma - 1)]^k.
\]
\end{proof}

Now for a useful bound on the $g$ function.

\begin{lemma}
For $\epsilon \geq 0$,
\[
\max\{g((1+\epsilon)^{-1}),g(1 + \epsilon))\} \leq 
  \exp(-(1/2)\epsilon^2(1-(4/3)\epsilon)).
\]
\end{lemma}

\begin{proof}
Note 
$\ln(g(1+\epsilon)) = \ln(1+\epsilon) - \epsilon 
  = -\epsilon^2/2 + \epsilon^3/3 - \cdots$ which is an alternating series for 
  $\epsilon \geq 0$.  Similarly,
$\ln(g((1+\epsilon)^{-1})) = -\ln(1+\epsilon) - \epsilon/(1+\epsilon)
 = -\epsilon^2/2 + (2/3)\epsilon^3 - \cdots$ which is also an alternating
series for $\epsilon \geq 0$.
\end{proof}

\begin{lemma}
For $\epsilon \in (0,3/4)$, when
\[
k \geq 2 \epsilon^{-2}(1 - (4/3)\epsilon)^{-1} \ln(2\delta^{-1}),
\]
$\prob(|(\hat p /p) - 1| > \epsilon) < \delta$.
\end{lemma}

\begin{proof}
Let $X \sim \gammadist(k,k-1)$.  Then $(\hat p/p) - 1 \sim (1/X) - 1$,
so 
\begin{align*}
\prob(|(\hat p /p) - 1| > \epsilon) &= \prob(|(1/X) - 1| > \epsilon) \\
 &= \prob(-\epsilon > (1/X) - 1) + \prob((1/X) - 1 > \epsilon) \\
 &= \prob((1 - \epsilon)^{-1} < X) + \prob(X < (1 + \epsilon)^{-1}) \\
 &= \prob(X > \gamma_1 \mean[X]) + \prob(X < \gamma_2 \mean[X]), \\
 &\leq g(\gamma_1)^k + g(\gamma_2)^k
\end{align*}
where $\gamma_1 = [(k-1)/k](1 - \epsilon)^{-1}$ and 
$\gamma_2 = [(k-1)/k](1 + \epsilon)^{-1}$.  For $k \geq \epsilon^{-2}$,
$(k-1)/k \geq 1 - \epsilon^2$, so $\gamma_1 \geq 1 + \epsilon$, and 
$\gamma_2 \leq (1 + \epsilon^{-1})$.
Since $g(x) = x/\exp(x-1)$ is increasing when $x < 1$, and decreasing when
$x > 1$, it holds that 
$g(\gamma_1) \leq g(1+\epsilon)$ and $g(\gamma_2) \leq g((1+\epsilon)^{-1})$.

Using the previous lemma,
\begin{align*}
\prob(|(\hat p /p) - 1| > \epsilon) 
 &\leq 2\exp(-\epsilon^{-2}k(1 - (4/3)\epsilon))
\end{align*}
and the result follows.
\end{proof}

\section{Lower bound on running time}
\label{SEC:lowerbound}

The new algorithm intentionally introduces random smoothing to make
the estimate easier to analyze.  For a fixed number of flips,
a sufficient statistic for the mean
of a Bernoulli random variable is 
the number of times the coin came up heads.  Call this number $S$.

For $k$ flips of the coin, $S$
will be a binomial random variable
with parameters $n$ and $p$.  Then $\hat p_n = S/n$ is the 
unbiased estimate of $p$.
By the Central Limit Theorem, $\hat p_n$ will be 
approximately normally distributed with
mean $p$ and standard deviation $\sqrt{p(1-p)/n}$.  Therefore (for small $p$),
$\hat p_n/p$ will be approximately normal with mean 1 and standard
deviation $1/\sqrt{p n}$.  Let $Z$ denote such a normal.  Then well known
bounds on the tails of the normal distribution give
\[
\frac{\exp(-\epsilon^2 p n /2) }{\sqrt{2\pi}}
 \left(\frac{1}{\epsilon^2 p n} - \frac{1}{(\epsilon^2 p n)^3}\right) 
 \leq 
\prob(Z > 1 + \epsilon) \leq 
\frac{\exp(-\epsilon^2 p n /2) }{\sqrt{2\pi}}
 \left(\frac{1}{\epsilon^2 p n}\right).
\]

Therefore, to get $\prob(Z > 1 + \epsilon) < \delta / 2$ requires about
$2 \epsilon^{-2} p^{-1} \ln(2 \delta^{-1})$ samples.  
A bound on the lower tail may be found in a similar fashion.
Since only about this many
samples are required by the algorithm of Section~\ref{SEC:algorithm}, 
the constant of 2 in front is 
most likely the best possible.

To actually prove a lower bound, follow the approach 
of~\cite{dagumklr2000} that uses Wald's sequential probability ratio test.
Consider the problem of testing hypothesis
$H_0:p = p_0$ versus $H_1:p = p_1$, where $p_1 = p_0/(1 + \epsilon)^2$.  
Suppose there is an approximating scheme that approximates $p$
within a factor of $1 + \epsilon$ with chance at least $1 - \delta / 2$
for all $p \in [p_1,p_0]$ using $T$ flips of the coin. 
Then take the estimate $\hat p$ and 
accept $H_0$ (reject $H_0$) if $\hat p \geq p_1(1 + \epsilon)$ and 
accept $H_1$ (reject $H_1$) if $\hat p \leq p_1(1 + \epsilon)$.

Then let $\alpha$ be the chance that $H_0$ is rejected even though it is
true, and $\beta$ be the chance that $H_1$ is accepted even though it is
false.  From the properties of the approximation scheme, 
$\alpha$ and $\beta$ are both at most $\delta/2$.

Wald presented the sequential probability ratio test for testing 
$H_0$ versus $H_1$, and showed that it minimized the expected number of 
coin flips among all tests with the type I and II error probabilities
$\alpha$ and $\beta$~\cite{wald1947}.  This result was formulated as
Corollary 7.2 in~\cite{dagumklr2000}.

\begin{fact}[Corollary 7.2 of~\cite{dagumklr2000}]
\label{FCT:wald}
If $T$ is the stopping time of any test of $H_0$ versus $H_1$ with 
error probabilities $\alpha$ and $\beta$ such that $\alpha + \beta = \delta$, 
then 
\[
\mean[T|H_0] \geq -(1 - \delta)\omega_0^{-1}\ln((2-\delta)\delta^{-1}).
\]
where $\omega_0 = \mean[\ln(f_1(X)/f_0(X))]$ with
$X \sim \bern(p_0)$, $f_0(x) = p_0 \ind(x = 1) + (1 - p_0) \ind(x = 0),$ 
and $f_1(x) = p_1 \ind(x = 1) + (1 - p_1) \ind(x = 0)$.
\end{fact}

This gives the following lemma for $\bern(p)$ random variables.

\begin{lemma}
Fix $\epsilon > 0$ and $\delta \in (0,1)$.
Let $T$ be the stopping time of any $(1 + \epsilon,\delta/2)$ approximation
scheme that applies to $X_i \sim \bern(p)$ for all $p \in [0,1]$.  Then
\[
\mean[T] \geq \lowerbound.
\]
\end{lemma}

\begin{proof}
As noted above,
using the approximation scheme with $\epsilon$ and $\delta / 2$ 
to test if $p_0 = p$ or 
$p_1 = p_0 / (1 + \epsilon)^2$ gives $\alpha \leq \delta/2$ and 
$\beta \leq \delta/2$.
Here
\begin{align*}
\omega_0 &= p_0(\ln(p_1/p_0)) + (1 - p_0)\ln((1 - p_1)/(1 - p_0)) \\
 &= p_0 [\ln(p_1/p_0) + (1/p_0 - 1)\ln((1 - p_1)/(1 - p_0))]  \\
 &= p_0 \ln\left[\frac{p_1(1 - p_1)^{1/p_0 - 1}}{p_0(1 - p_0)^{1/p_0 - 1}}\right].
\end{align*}

Consider a function of the form 
$g(x) = x(1 - x)^{1/c - 1}$ where $c$ is a constant.  Then $g(x) > 0$ for 
$x \in (0,1)$, and  
$g'(x) = g(x)x^{-1} - (1/c - 1) g(x) (1 - x)^{-1}$, which gives
\[
g'(x) > 0 \Leftrightarrow x^{-1} - (1/c - 1)(1 - x)^{-1} 
 \Leftrightarrow x < c.
\]
Hence for all $p_0 > p_1$, $\ln(p_1(1 - p_1)^{1/p_0 - 1})$ is 
strictly increasing in 
$p_1$.  Setting $p_1 = p_0$ gives $\omega_0 = 0$, so $\omega_0 < 0$ 
for $0 < p_1 < p_0 \leq 1$.

Using $\alpha + \beta \leq \delta$ and $\omega_0 < 0$ in Fact~\ref{FCT:wald}
gives
\begin{align*}
\mean[T] &\geq -\omega_0^{-1}(1 - \delta)\ln((2-\delta)\delta^{-1}).
\end{align*}

Since $\ln(1 + x) = x - x^2/2 + \cdots$ is alternating and
decreasing in magnitude for $x \in (0,1)$:
\[
\ln\left(\frac{p_1}{p_0}\right) = \ln\left(\frac{1}{(1 + \epsilon)^2}\right)
 = -2\ln(1 + \epsilon) \geq -2\epsilon.
\]
Also, since 
$1 - (1 + \epsilon)^{-2} = (2\epsilon + \epsilon^2)/(1 + \epsilon)^2.$
\begin{align*}
\left(\frac{1}{p_0} - 1\right)\ln\left(\frac{1 - p_1}{1 - p_0}\right) 
&= \left( \frac{1 - p_0}{p_0} \right) \ln\left(
  \frac{1 - p_0(1 + \epsilon)^{-2}}{1 - p_0} \right) \\
&= \left( \frac{1 - p_0}{p_0} \right) \ln
  \left( 1  +\frac{p_0(1 - (1 + \epsilon)^{-2})}{1 - p_0} \right) 
  \\
&= \left( \frac{1 - p_0}{p_0} \right)
  \left[
  \left( \frac{p_0(1 - (1 + \epsilon)^{-2})}{1 - p_0} \right)
  - \frac{1}{2} \left( \frac{p_0(1 - (1 + \epsilon)^{-2})}{1 - p_0} \right)^2
  \right] \\
&\geq \frac{2\epsilon + \epsilon^2}{(1 + \epsilon)^2} - 
  \frac{1}{2} \cdot \left[\frac{2\epsilon + \epsilon^2}{(1 + \epsilon)^2}
   \right]^2
   \cdot \frac{p_0}{1 - p_0}.
\end{align*}

For $p_0 \leq 1/2$, $p_0/(1 - p_0) \leq 1$ and the last factor of the 
second term can be removed.  Putting the bounds on the 
terms of $\omega_0$ together,
\begin{align*}
\omega_0 &\geq p_0\left[-2\epsilon
 +  \frac{2\epsilon + \epsilon^2}{(1 + \epsilon)^2} - 
  \frac{1}{2} \cdot \left(\frac{2\epsilon + \epsilon^2}{(1 + \epsilon)^2}
  \right)^2 \right] \\
 &= p_0 \frac{-5\epsilon^2(1 + 2\epsilon + (3/2)\epsilon^2 + (2/5)\epsilon^3)}
  {(1 + \epsilon)^4} \\
 &\geq - p_0 5\epsilon^2 / (1 + 2\epsilon).
\end{align*}
The last inequality follows from the fact that for $\epsilon > 0$,
\[
(1 + 2\epsilon)(1 + 2\epsilon + (3/2)\epsilon^2 + (2/5)\epsilon^3) 
 \leq (1 + \epsilon)^4.
\]
\end{proof}

\section{Extension to $[0,1]$ random variables}

A well known trick allows extension of the algorithm to $[0,1]$ random
variables with mean $\mu$, rather than just Bernoulli's.

\begin{lemma}
Let $W$ be a $[0,1]$ random variable with mean $\mu$.  Then
for $U \sim \unif([0,1])$, $\prob(U \leq W) = \mu$.
\end{lemma}

\begin{proof} For $U \sim \unif([0,1])$ and $W \in [0,1]$,
\[
\prob(U \leq W) = \int_{w = 0}^1 \prob(U \leq w) \ dF(w) 
 = \int_{w=0}^1 w \ dF(w) = \mean[W].
\]
\end{proof}

Therefore the algorithm of Section~\ref{SEC:algorithm} can be applied
to any $[0,1]$ random variable at the cost of one uniform on $[0,1]$
per draw of the random variable.

\section{Conclusions}

A new algorithm for estimating the mean of $[0,1]$ variables is given
with the remarkable property that the relative error in the estimate
has a distribution independent of the quantity to be estimated.  
The estimate is unbiased.
To obtain an estimate which has absolute relative error $\epsilon$ 
with probability at least $1 - \delta$ requires at most 
$2\epsilon^{-2} (1 - (14/3)\epsilon)^{-1} p^{-1} \ln(2\delta^{-1})$ samples.  The
factor of 2 is an improvement over the factor of $4(e - 2)$ 
in~\cite{dagumklr2000}.
Informal Central Limit Theorem arguments indicate that this factor of 2 in
the running time is the best possible. 
The provable lower bound
on the constant is improved from the $(1/4)e^{-2} \approx 0.0338$ 
of~\cite{dagumklr2000} to $1/5$ for $\{0,1\}$ random variables.

\bibliographystyle{plain}


\end{document}